\documentclass[11pt]{amsart}
\usepackage{geometry}                
\usepackage[parfill]{parskip}    
\usepackage{graphicx}
\usepackage{caption}
 \usepackage{subcaption}
\usepackage{amssymb}
\usepackage{epstopdf}
\usepackage[all,knot,poly]{xy}
\numberwithin{equation}{section}
\newtheorem{theorem}{Theorem}[section]

\newtheorem{lemma}{Lemma}[section]
\newtheorem{corollary}{Corollary}[section]

\title{Predicting the number and type of twist sites in a rational knot or link}
\author{Mark E. Kidwell}
\email{mek@usna.edu}
\author{Kerry M. Luse}
\email{lusek@trinitydc.edu}
\date{\today}                                     

\begin{document}
\maketitle

\begin{abstract}
A rational knot or link can be put into a standard alternating format which has horizontal and vertical twist sites (double helices).  The number and type of these twist sites are determined by terms of next-to-highest $z$-degree in Kauffman's regular isotopy invariant $\Lambda (a,z)$.  In particular, for a knot or link with $c$ crossings, the coefficient of the $z^{c-2}$ term is equal to the number of twist sites in its standard diagram.  Furthermore, the coefficients of the $a^{-2}z^{c-2}$ and $a^2z^{c-2}$ terms count the number of left-turning and right-turning twist sites, respectively.
\end{abstract}

\section{Introduction}

In attempting to extend some results of Abe, \cite{Abe}, about Kauffman's clock moves, \cite{FKT}, it became necessary to distinguish rational knots and links from more complicated types.  Using Conway's continued fraction rule \cite{Conway}, it is possible to put a rational tangle into a standard format with a certain number of twist sites (double helices) that alternate between horizontal and vertical.  (Twist sites with a single crossing are somewhat troublesome, but can be classified as horizontal or vertical by their relative position in the tangle.)

As is well known to pipe threaders, there are two types of helices, those with positive torsion and those with negative torsion. See Figure \ref{fig:twist type}.  This fact, which also applies to double helices, is independent of any orientation assigned to the strands.  A given twist site can contribute positive or negative writhe to its tangle depending on how the strands are oriented.  In an alternating diagram of a rational tangle, the twist sites must alternate between double helices of positive and negative torsion.

\begin{figure}[h]
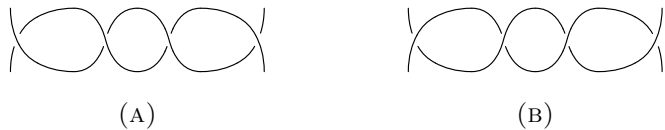

        \centering
         \begin{subfigure}[h]{0.35\textwidth}
                \centering
$$\xygraph{
!{0;/r2.0pc/:}
!{\hunder}
!{\hcrossneg}
!{\hcrossneg}
!{\hunder-}
} $$
                 \caption{}
                 \label{fig:positive twist}
         \end{subfigure}%
         \begin{subfigure}[h]{0.35\textwidth}
                 \centering
                $$ \xygraph{
!{0;/r2.0pc/:}
!{\hover}
!{\hcross}
!{\hcross}
!{\hover-}
}$$
                 \caption{}
                 \label{fig:negative twist}
         \end{subfigure}
\caption{Positive and negative twist types.}
\label{fig:twist type}
 \end{figure}

We show in this paper that certain terms in Kauffman's two-variable, regular isotopy invariant $\Lambda(a,z)$ predict the number and type of twist sites in a rational knot or link in standard alternating form.  It has been known since the 1980's \cite{Thi} that, for a reduced, alternating diagram with $c$ crossings, the terms of $\Lambda$ of highest$z$-degree have $z$-exponent $c-1$ and depend only on the underlying Conway basic polyhedron of the diagram.  In particular, for an algebraic knot or link (and rationals are a subclass), these terms are $a^{-1}z^{c-1}+az^{c-1}$.

Most of our attention will be focused on the next-highest $z$-terms $a^{-2}z^{c-2}$, $z^{c-2}$, and $a^2z^{c-2}$.  We shall call the coefficients of these three terms $u_-$, $u_0$, and $u_+$.  According to Thistlethwaite, no term in $\Lambda$ can have the sum of its $z$-exponent and the absolute value of its $a$-exponent greater than $c$ \cite{Thi}. Thistlethwaite also proves that $u_-$, $u_0$, and $u_+$ are non-negative.  We will prove that the coefficient $u_0$ of $z^{c-2}$ for a rational knot or link equals the number of twist sites in its standard diagram.

Before we can discuss $u_-$ and $u_+$, we must mention that the Hopf link is ambiguous as to whether its twist sites are left-turning or right-turning (see Figure \ref{fig: right and left hopf} ).  In this case, $c=2$, there is one twist site, and the only term in $\Lambda$ with $z$-exponent zero is 1.  That is, $u_-=u_+=0$.  In all other cases involving alternating knots, $u_-+u_+=u_0$, as proved in \cite{Thi}.  We will show that $u_+$ and $u_-$ are equal to the number of right-turning and left-turning twist sites for a rational knot or link with three or more crossings in standard format.

\begin{figure}[h]
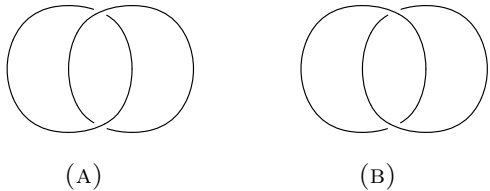

\centering
\begin{subfigure}[h]{0.25\textwidth}
\centering
$$\xygraph{
!{0;/r2.0pc/:}
!{\vunder}
!{\vunder-}
[uur]!{\hcap[2]}
[l]!{\hcap[-2]}
}$$
\caption{}
\label{fig:left hand hopf}
\end{subfigure}
\begin{subfigure}[h]{0.25\textwidth}
$$\xygraph{
!{0;/r2.0pc/:}
!{\vover}
!{\vover-}
[uur]!{\hcap[2]}
[l]!{\hcap[-2]}
}$$
\caption{}
\label{fig:right hand hopf}
\end{subfigure}
\caption{The left- and right-handed Hopf links are equivalent}
\label{fig: right and left hopf}
\end{figure}




\section{Standard format for a rational tangle}

We will draw our rational tangles in a herringbone pattern starting at the northwest (NW) quadrant of the tangle and ending at the southeast (SE) quadrant (Figure \ref{fig:standard knot}).  In addition:

\begin{enumerate}
\item The first and last twist sites have at least two crossings.\label{tangle cond 1}
\item The twist sites alternate between horizontal and vertical.  This rule defines whether a twist site with a single crossing is horizontal or vertical. \label{tangle cond 2}
    \item All horizontal twist sites are left-turning and all vertical twist sites are right-turning. \label{tangle cond 3}
    \item The last (SE) twist site is horizontal. \label{tangle cond 4}
    \end{enumerate}

\begin{figure}[h]
   \centering
    \includegraphics[width=.65\textwidth]{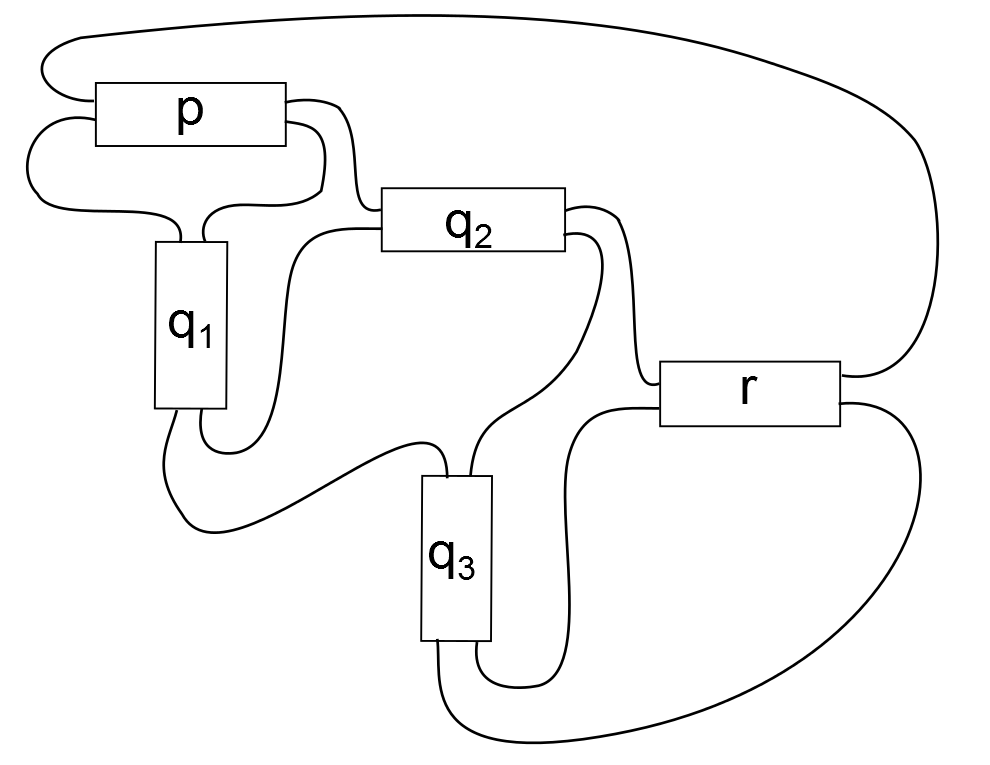}
 \caption{A standard diagram of a rational knot.}\label{fig:standard knot}
 \end{figure}

These rules have a number of simple consequences. Rule \ref{tangle cond 3} ensures that our tangle diagrams are alternating.  At every crossing, the overcrossing segment has positive slope.  In keeping with Rule \ref{tangle cond 4}, we will always use the ``numerator closure,'' $N(T)$, (NW strand joined to NE, SW strand joined to SE) when making our tangles into rational knots or links.  By rules \ref{tangle cond 2} and \ref{tangle cond 3}, the number of left-turning twist sites will never differ from the number of right-turning twist sites by more than one.  By \ref{tangle cond 4}, left-turning twist sites will never be in the minority.

We will write the Conway notation \cite{Conway} for a rational tangle with three or more twist sites as $pq_1q_2\ldots q_kr$, where $p\geq2$, $r\geq2$, and $q_i\geq1$ for $i=1,\ldots,k$.  There are $k+2$ twist sites in such a tangle.  Rational tangles with one or two twist sites are exceptional, and will be treated separately.

\section{Properties of Kauffman's $\Lambda$ polynomial}

We follow the conventions of Kauffman's defining paper \cite{RegIs} and of \cite{Thi}, namely:

\begin{enumerate}
\item A simple closed curve in the plane has $\Lambda$ invariant 1. \label{simple}
\item (The four term skein relation) \label{skein}
$$\Lambda_{D_+}+\Lambda_{D_-}=z(\Lambda_{D_0}+\Lambda_{D_\infty})$$
See Figure \ref{fig:skein relation}.
\item (The loop relation)\label{loop}

\begin{itemize}
\item [a.]
$\Lambda\bigl(\hspace{.5cm}
\xy
(0,0)*{\xygraph{
*\xycircle<10pt>{{.}}}};
(-1,2)*{\xygraph{
!{0;/r1.0pc/:}
!{\hcross}
!{\hcap}
}}
\endxy
\hspace{1.25cm}
\bigr)
=a\Lambda\bigl(\hspace{.5cm}
\xy
(0,0)*{\xygraph{
*\xycircle<10pt>{{.}}}};
(-1,2)*{\xygraph{
!{0;/r1.0pc/:}
!{\huntwist}
!{\hcap}
}}
\endxy
\hspace{1.25cm}
\bigr)$
\item [b.]$\Lambda\bigl(\hspace{.5cm}
\xy
(0,0)*{\xygraph{
*\xycircle<10pt>{{.}}}};
(-1,2)*{\xygraph{
!{0;/r1.0pc/:}
!{\hcrossneg}
!{\hcap}
}}
\endxy
\hspace{1.25cm}
\bigr)=a^{-1}\Lambda\bigl(\hspace{.5cm}
\xy
(0,0)*{\xygraph{
*\xycircle<10pt>{{.}}}};
(-1,2)*{\xygraph{
!{0;/r1.0pc/:}
!{\huntwist}
!{\hcap}
}}
\endxy
\hspace{1.25cm}
\bigr)$
\end{itemize}

\end{enumerate}

\begin{figure}[h]
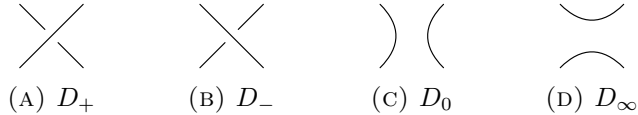

         \begin{subfigure}[t]{0.15\textwidth}
\centering
$\xygraph{
!{0;/r2.0pc/:}
!{\xunderh}
} $
                 \caption{$D_+$}
                 \label{fig: D_+}
\end{subfigure}
\begin{subfigure}[t]{0.15\textwidth}
\centering
                $ \xygraph{
!{0;/r2.0pc/:}
!{\xoverh}
}$
                 \caption{$D_-$}
                 \label{fig: D_-}
         \end{subfigure}
    \begin{subfigure}[t]{0.15\textwidth}
\centering
                $ \xygraph{
!{0;/r2.0pc/:}
!{\xunoverv}
}$
                 \caption{$D_0$}
                 \label{fig: D_0}
         \end{subfigure}
    \begin{subfigure}[t]{0.15\textwidth}
\centering
                $ \xygraph{
!{0;/r2.0pc/:}
!{\xunoverh}
}$
                 \caption{$D_{\infty}$}
                 \label{fig: D_infty}
         \end{subfigure}
\caption{The four terms in the skein relation for Kauffman's $\Lambda$ polynomial.}\label{fig:skein relation}
 \end{figure}

Rule \ref{loop} shows that $\Lambda$ is changed by type I Reidemeister moves, so it can only be a regular isotopy (preserved by Reidemeister II and III) invariant.

The symmetry of $D_+$ and $D_-$ , and of $D_0$ and $D_{\infty}$, makes it difficult to distinguish them.  By our conventions for standard format, all crossings are of the form $D_+$.  If $D_+$ is an alternating diagram, then $D_-$ will have a bridge that overcrosses at three consecutive crossings.  According to \cite{Thi} and \cite{Kid}, the largest exponent of $z$ in $\Lambda_{D_-}$ is at most $c-3$, so this polynomial will not contribute to the $z^{c-2}$ or higher terms of $D_+$.  In a slight deviation from \cite{Thi}, we define $\tilde{\Lambda}(a,z)=u_-a^{-2}z^{c-2}+u_0z^{c-2}+u_+a^2z^{c-2}+a^{-1}z^{c-1}+az^{c-1}$.  This truncated invariant satisfies the skein relation
\begin{equation}
\tilde{\Lambda}_{D_+}=z(\tilde{\Lambda}_{D_0}+\tilde{\Lambda}_{D_{\infty}})\label{eq: truncated skein}
\end{equation}

It is possible and necessary to distinguish the two types of smoothings at a crossing along a twist site with at least two crossings.  We say that a smoothing is \textbf{\em{axial}} if it reduces a twist site of $n$ crossings to a twist site of $n-1$ crossings.  We say that a smoothing is \textbf{\em{cross-sectional}} if it cuts across the twist site and creates a loop.  See Figure \ref{smoothings}.  It then follows that in a horizontal twist site, $D_0$ is a cross-sectional smoothing and $D_{\infty}$ is an axial smoothing, while for a vertical twist site, $D_0$ is an axial smoothing and $D_{\infty}$ is a cross-sectional smoothing.  This last observation can be used to classify smoothings of one-crossing twist sites as axial or cross-sectional.  However, we will avoid smoothing a one-crossing twist site as much as possible. A large part of our argument will involve specifying when a given rational knot or link has the same $\tilde{\Lambda}$ invariant (up to a power of $z$) as an axial smoothing at some crossing and when the cross-sectional smoothing makes a contribution to $\tilde{\Lambda}$.

 \begin{figure}[h]
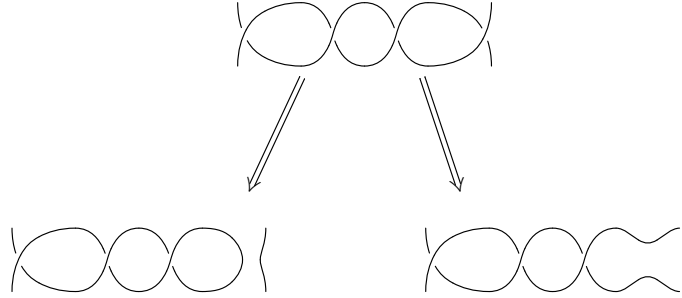


\[
\xy
(-5,30)*{ \xygraph{
!{0;/r2.0pc/:}
!{\hover}
!{\hcross}
!{\hcross}
!{\hover-}
}}="x";
(12,20)*{\xygraph{}}="a";
(28,20)*{\xygraph{}}="b";
(-35,0)*{ \xygraph{
!{0;/r2.0pc/:}
!{\hover}
!{\hcross}
!{\hcross}
!{\hunover-}
}}="y";
(5,5)*{\xygraph{}}="c";
(33,5)*{\xygraph{}}="d";
(20,0)* {\xygraph{
!{0;/r2.0pc/:}
!{\hover}
!{\hcross}
!{\hcross}
!{\huntwist}
}}="z";
{\ar@2{>} "a";"c"};
{\ar@2{>} "b";"d"};
\endxy
\]
\vspace{.2cm}
\caption{A cross-sectional smoothing (left) and an axial smoothing (right) performed on the right most crossing. }
\label{smoothings}
\end{figure}

\section{Main results}

We will say that a rational tangle $pq_1q_2\ldots q_kr$ is \textbf{\em{minimal}} if $p=2$, $r=2$, and $q_i=1$ for $i=1,2,\ldots k$.  Otherwise, we say that the tangle (and a particular twist site) has \textbf{\em{extra crossings}}.

\begin{lemma} \label{lem: T to T'}
If a rational tangle $T$ in standard format with $c$ crossings has extra crossings, then there is a tangle $T'$ with $c-1$ crossings, still in standard format with the same number of twist sites as $T$, such that $T'$ is obtained from $T$ by an axial smoothing.
\end{lemma}

\begin{proof}
Simply perform the axial smoothing at a twist site with extra crossings.
\end{proof}

We point out, in anticipation of future research, that Lemma \ref{lem: T to T'} applies not only to rational knots and links, but to any knot or link that is decomposed into rational tangles.

\subsection{Rational tangles with one or two twist sites}

We begin with the $\Lambda$-invariant of the Hopf link, as given in \cite{knotinfo}:

$$\begin{array}{lccc}
\Lambda(a,z)=&-az^{-1}& & +az \\
&&+1&\\
&-a^{-1}z^{-1}&&+a^{-1}z
\end{array}$$

and thus

$$\begin{array}{lcc}
\tilde{\Lambda}(a,z)=&&az\\
&+1&\\
&&+a^{-1}z.
\end{array}$$

The middle term of $z^{c-2}=z^0$ is 1, and, in standard format, the link has one twist site.  As noted earlier, the handedness of this twist site is ambiguous.

We move on to a left-turning trefoil.  The axial smoothing of one of its crossings is the Hopf link, and the cross-sectional smoothing is an unknot with two crossings that give it a $\Lambda$-invariant of $a^2$ (see Figure \ref{fig:trefoil smoothing}).By \eqref{eq: truncated skein}, we have $\tilde{\Lambda}_{\textrm{tref}}(a,z)=z(1+a^{-1}z+az+a^2)$, or

$$\begin{array}{lcc}
\tilde{\Lambda}_{\textrm{tref}}(a,z)=&a^2z&\\
&&+az^2\\
&+z&\\
&&+a^{-1}z^2
\end{array}$$

\begin{figure}[h]
   \centering
    \includegraphics[width=.75\textwidth]{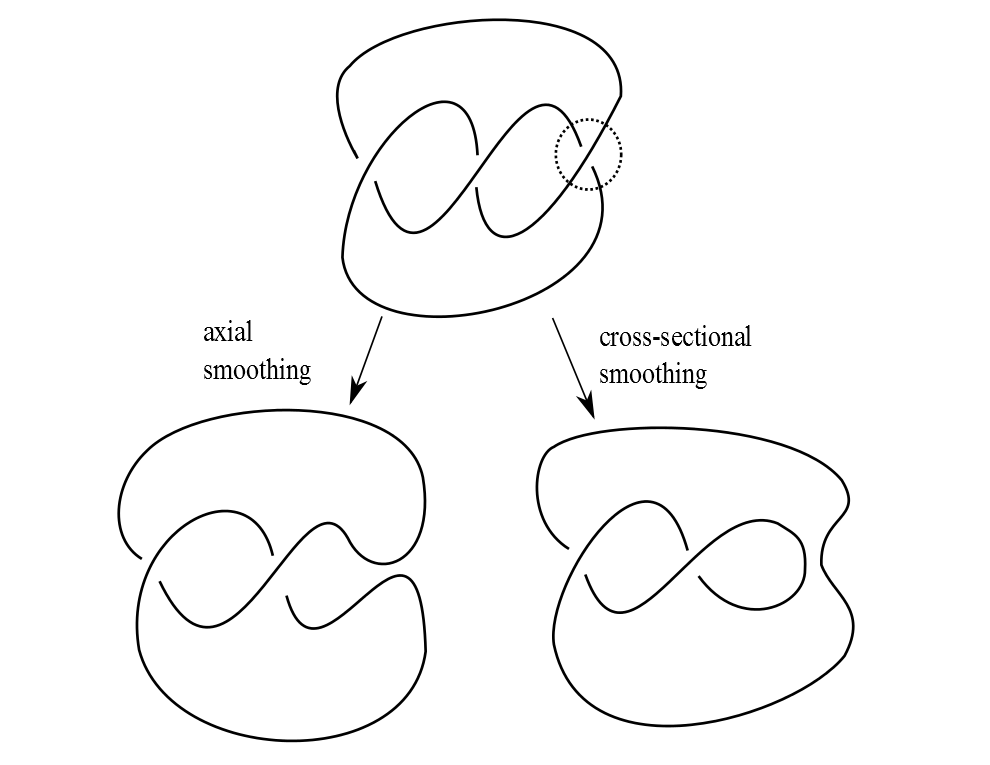}
 \caption{The axial and cross-sectional smoothings of a trefoil knot.}\label{fig:trefoil smoothing}
 \end{figure}

\begin{lemma} Let $K$ be a knot or link with a single left-turning twist site and $c\geq4$ crossings.
Then $\tilde{\Lambda}$ is identical, up to powers of $z$, to $\tilde{\Lambda}$ for the left-turning trefoil.
\end{lemma}

\begin{proof}
 In this case, $\tilde{\Lambda}$ will have terms of degree $c-1$ and $c-2\geq 2$.  The cross-sectional smoothing is again an unknot with $\Lambda=a^{c-1}$.  Even multiplied by $z$ in \eqref{eq: truncated skein}, this term makes no contributions to the $\tilde{\Lambda}$-invariant of the $c$-crossing knot.  We can thus smooth our way back to the trefoil with no change in $\tilde{\Lambda}$ except added powers of $z$.
\end{proof}

That is: when we have a single left-turning twist site with at least three crossings, $u_+=1$ (the number of left turning twist sites), $u_0=1$ (the total number of twist sites), and $u_-=0$ (the number of right-turning twist sites).

We move on to rational tangles of the form $pr$ (two twist sites).  The simplest such tangle satisfying our rules is 22, which closes to the figure-eight knot.  According to \cite{knotinfo}:

$$\begin{array}{lcc}
\tilde{\Lambda}_{\textrm{fig 8}}(a,z)=&a^2z^2&\\
&&+az^3\\
&+2z^2&\\
&&+a^{-1}z^3\\
&a^{-2}z^2
\end{array}$$

The following lemma shows that this pattern holds for the numerator closure of any tangle $pr$.

\begin{lemma} Let $K$ be a knot or link corresponding to the numerator closure of the rational tangle $pr$ with $p\geq2$ and $r\geq2$.  Then $\tilde{\Lambda}_K$ is identical up to powers of $z$ to $\tilde{\Lambda}_{\textrm{fig 8}}$.
\end{lemma}

\begin{proof}
The case where $p=2$ and $r=2$ is just the case of the figure eight knot.  The remaining cases have either $p>2$ or $r>2$, or both. The invariant $\tilde{\Lambda}$ of such a knot or link will have non-zero $z$-exponents at the powers $p+r-1$ and $p+r-2$.  If we perform the cross-sectional smoothing of a $p$-crossing with $p\geq3$, we get a diagram of $r$ with extra crossings (see Figure \ref{fig:N43} for example).  The highest $z$-power in this smoothed knot or link will be $z^{r-1}$, but when we multiply by $z$ as in \eqref{eq: truncated skein} we get $z^r$.  Since $3\leq p$, $r+3\leq p+r$, $r+1\leq p+r-2$ and so $r<p+r-2$ and the cross-sectional smoothing makes no contribution to the $\tilde{\Lambda}$-invariant of $N(pr)$.  A similar analysis holds if we perform a smoothing at an $r$-site with $r\geq3$.
\end{proof}

\begin{figure}[h]
   \centering
    \includegraphics[width=.75\textwidth]{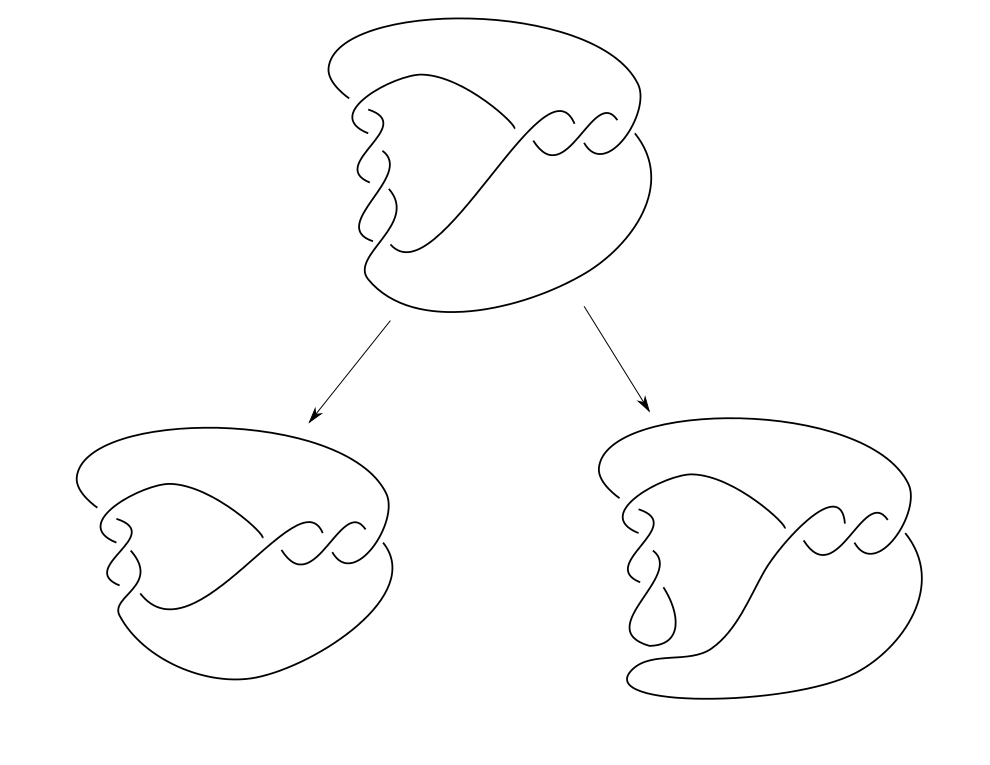}
 \caption{The cross-sectional and axial smoothings of the knot $N(43)$.}\label{fig:N43}
 \end{figure}

In all cases, $u_+=1$ (the number of left-turning twist sites), $u_0=2$ (the total number of twist sites), and $u_-=1$ (the number of right turning twist sites).

\subsection{Rational tangles with three or more twist sites}

\begin{lemma}\label{lem: T' connect sum}
Let $T=pq_1q_2\ldots q_kr$ be a tangle in standard format with at least three twist sites.  Let $T'$ be a tangle obtained by performing a cross-sectional smoothing on one crossing in an interior twist site of $T$.  Then $N(T')$ is the connected sum of two knots or links.
\end{lemma}

\begin{proof}

See Figure \ref{fig:general smoothing}.  Suppose the cross-sectional smoothing is performed on the $q_i$ twist site.  There is a simple closed curve in the plane that intersects $T$ in four points and separates $pq_1\ldots q_i$ from $q_{i+1}\ldots q_kr$.  (If $i=k$, the $r$-twist site will stand alone.)  After the cross-sectional smoothing, this simple closed curve can be made to intersect $T$ in two points, defining the connected sum.
\end{proof}

\begin{figure}[h]
   \centering
    \includegraphics[width=.75\textwidth]{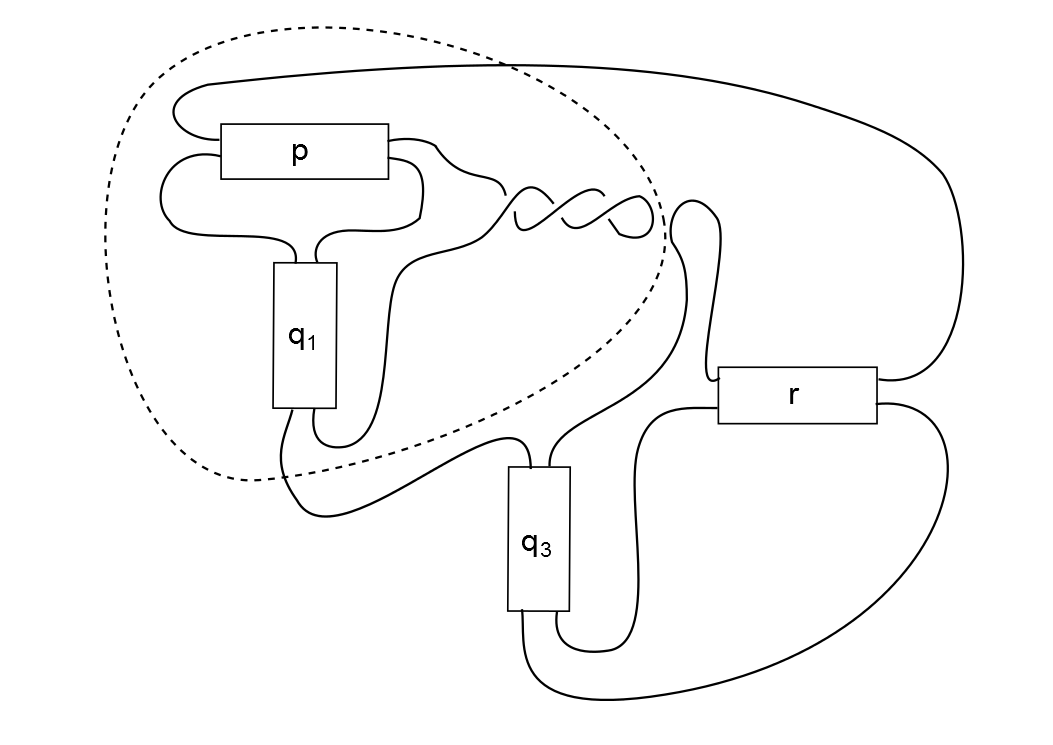}
 \caption{The cross-sectional smoothing of a rational knot in standard form.}\label{fig:general smoothing}
 \end{figure}

Recall that the $\Lambda$-invariant of a connected sum $K' \# K''$ is the product of the $\Lambda$-invariants of its factors.  Suppose $K'$ has $c_1$ crossings and $K''$ has $c_2$ crossings, both reduced and alternating.  For $\tilde{\Lambda}$ we have:

\[
\left( \begin{array}{cc}
u'_+a^2z^{c_1-2}&\\
&az^{c_1-1}\\
+u'_0z^{c_1-2}&\\
&+a^{-1}z^{c_1-1}\\
+u'_-a^{-2}z^{c_1-2}&
\end{array}
\right)
*
\left( \begin{array}{cc}
u''_+a^2z^{c_2-2}&\\
&az^{c_2-1}\\
+u''_0z^{c_2-2}&\\
&+a^{-1}z^{c_2-1}\\
+u''_-a^{-2}z^{c_2-2}&
\end{array}
\right)
=
\]
\[
\cdots
\begin{array}{cc}
+(u'_++u_+^{''})a^3z^{c_1+c_2-3}&\\
&a^2z^{c_1+c_2-2}\\
+(u'_++u'_0+u''_++u''_0)az^{c_1+c_2-3}&\\
&+2z^{c_1+c_2-2}\\
+(u'_0+u'_-+u''_0+u''_-)a^{-1}z^{c_1+c_2-3}&\\
&a^{-2}z^{c_1+c_2-2}\\
+(u'_-+u''_-)a^{-3}z^{c_1+c_2-3}&\\
\end{array}
\]

In particular, if the connected sum has $c$ crossings, then the highest $z$-degree of its $\tilde{\Lambda}$-invariant is $c-2$.

Let the tangle $T$ be in the standard format $pq_1\ldots q_kr$.  We first want to show that if $T$ has $j$ extra crossings, then $\tilde{\Lambda}_T(a,z)=z^j\tilde{\Lambda}_m(a,z)$ where $m$ is the minimal tangle $21^k2$.  The cases $p\geq3$ and $r\geq3$ proceed as in the case of tangles with two twist sites.  Now suppose $q_i\geq2$.  We wish to show that $\tilde{\Lambda}_T(a,z)=z\tilde{\Lambda}_{T'}(a,z)$ where $T'=pq_1\ldots (q_i-1)\ldots q_kr$, which is equivalent to saying that the cross-sectional smoothing $T''$ at a $q_i$ twist site makes no contribution to $\tilde{\Lambda}_T$ if $q_i\geq2$.  If $T$ has $c$ crossings, then $T''$ loses one crossing at the smoothing and at least one other ($q_i\geq2$) when any loops along the $q_i$ twist site are un-looped.  We also know that $T''$ is a connected sum, so that $\tilde{\Lambda}_{T''}$ has terms of $z$-degree at most $c-4$.  But the terms of $\tilde{\Lambda}_T$ have $z$-degree $c-2$ and $c-1$ so that, even multiplied by $z$ in \eqref{eq: truncated skein}, $\tilde{\Lambda}_{T''}$ makes no contribution to $\tilde{\Lambda}_T$.

We iterate this argument until we are down to the minimal tangle $21^k2$.

\begin{theorem} A rational knot or link, $N(T)$, in standard format with $k+2$ twist sites and $c$ crossings has:

\[\begin{array}{lcc}
\tilde{\Lambda}(a,z)=&\left\lceil\frac{k+2}{2}\right\rceil a^2z^{c-2}&\\
& & +az^{c-1}\\
&+(k+2)z^{c-2}\\
& & +a^{-1}z^{c-1}\\
&+\left\lfloor\frac{k+2}{2}\right\rfloor a^{-2}z^{c-2}&
\end{array}
\]

\end{theorem}

\begin{proof}
The critical case is $N(21^k2)$.  We will smooth the last singleton vertex, which counts as a vertical twist site.  Figure \ref{fig:21112} illustrates that we lose two twist sites and one crossing when we perform the axial smoothing.  Call the resulting tangle $T'$.  As a result of this type of smoothing, 212 becomes 4, 2112 becomes 23, and for $k\geq 3$, $21^k2$ becomes $21^{k-2}3$.  $T'$ has $2+(k-2)+3=k+3$ crossings.  In all cases, the cross-sectional smoothings are connected sums with one component a Hopf link.  Call this tangle $T''$.  It has no loops, an alternating diagram, and $(k-1)+2+2=k+3$ crossings.  As explained above, the terms of highest $z$-degree in $\tilde{\Lambda}_{T''}$ are $(a^2+2+a^{-2})z^{k+1}$ (connected sum powers are reduced by 1).

\begin{figure}[h]
   \centering
    \includegraphics[width=.75\textwidth]{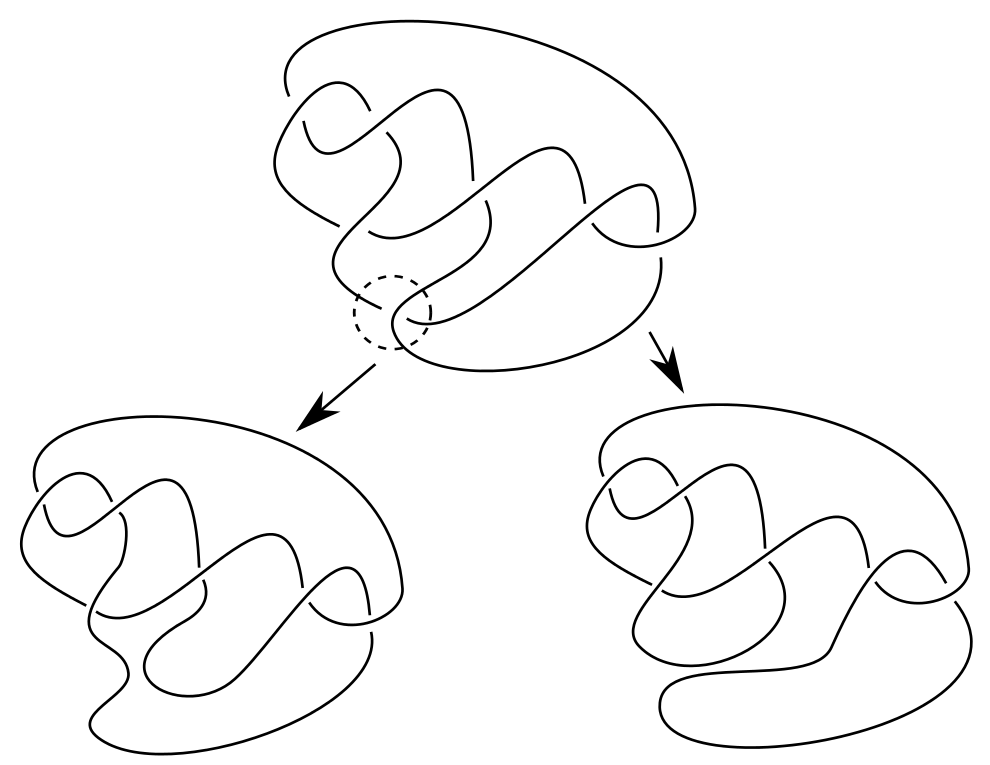}
 \caption{The axial smoothing of $N(21112)$ results in the knot $N(213)$.  The cross-sectional smoothing of $N(21112)$ results in the knot $N(22)\#N(2)$.}\label{fig:21112}
 \end{figure}

Meanwhile, we can assume by induction that:

\[\begin{array}{lcc}
\tilde{\Lambda}_{T'}(a,z)=&\left\lceil\frac{k}{2}\right\rceil a^2z^{k+1}&\\
& & az^{k+2}\\
&+kz^{k+1}\\
& & +a^{-1}z^{k+2}\\
&\left\lfloor\frac{k}{2}\right\rfloor a^{-2}z^{k+1}&
\end{array}
\]

Thus the highest $z$-powers of $\tilde{\Lambda}_{T''}$ are on the same level as the next-to-highest $z$-powers of $\tilde{\Lambda}_{T'}$.  By \eqref{eq: truncated skein}, we have for $\Lambda_T$ that $u_+=\left\lceil\frac{k}{2}\right\rceil+1$, $u_0=k+2$, and $u_-=\left\lfloor\frac{k}{2}\right\rfloor+1$.  If $k$ is even, say $k=2l$, then $\left\lceil\frac{k}{2}\right\rceil+1=l+1=\frac{k}{2}+\frac{2}{2}=\frac{k+2}{2}=\left\lceil\frac{k+2}{2}\right\rceil$.  If $k$ is odd, say $k=2l+1$, then $\left\lceil\frac{k}{2}\right\rceil+1=\left\lceil\frac{2l+1}{2}\right\rceil+1=\left\lceil l+\frac{1}{2}\right\rceil+1=l+1+1=\frac{k-1}{2}+2=\frac{k+3}{2}=\left\lceil\frac{k+3}{2}\right\rceil=\left\lceil\frac{k+2}{2}\right\rceil$.  Similarly, $u_-=\left\lfloor\frac{k}{2}\right\rfloor+1=\left\lfloor\frac{k+2}{2}\right\rfloor$.  The formula given in the theorem is verified for a minimal tangle.  Any tangle with extra crossings can be reduced to this case.

\end{proof}

We call a knot or link  \textbf{\em{top-heavy}} if $u_+\geq u_-$, \textbf{\em{bottom-heavy}} if $u_-\geq u_+$, and \textbf{\em{balanced}} if $u_+=u_-$.  It is a consequence of the choices we made when defining standard format that we get the top-heavy form of $\tilde{\Lambda}$ for knots with an odd number of twist sites.  The mirror image knot will have $\tilde{\Lambda}$ bottom-heavy.  In a rational knot or link with an even number of twist sites, $\tilde{\Lambda}$ will be balanced. (When taking a mirror image, $a$ is replaced by $a^{-1}$ in $\Lambda$.)  We could not have written this paper without the data provided by Livingston and Cha's excellent programs Knot Info and Link Info \cite{knotinfo},\cite{linkinfo}.  We are uncertain, however, how one form of the knot diagram was chosen as ``the'' form and the other was labeled as ``mirror.''  The situation for links is even more complicated.  In every case that we have examined, the Kauffman polynomial given corresponds to ``the'' knot diagram, but there is a slight preference for the bottom-heavy form.

Let us give a simple geometric consequence of our calculations:

\begin{corollary}
Let $T$ be a rational tangle in standard format with an odd number of twist sites.  Then $N(T)$ cannot be an amphicheiral knot or link.
\end{corollary}

\begin{proof}

Since $u_+=u_-+1$, the invariant $\tilde{\Lambda}$ does not have the necessary symmetry under $a\leftrightarrow a^{-1}$.

\end{proof}

\section{More complicated cases}

We hope to extend our work to algebraic knots and links that are not rational.  Our searches in Knot Info and Link Info have revealed at least two differences in the $\tilde{\Lambda}$ invariants:

\begin{enumerate}
\item The coefficient $u_0$ can exceed the number of twist sites by 1 (for example, $8_5$ and $8_{10}$ in Rolfsen notation) or 2 (for example, $10_{79}-10_{81}$).  We have not observed an excess of 3 or more, but cannot rule it out at this point.
\item In rational knots and links in standard format, we have $u_+=u_-$ or $u_+=u_-+1$.  This near-balance does not necessarily hold in non-rational knots or links.  For example, the non-rational link with smallest crossing number $6^3_1=L6a5=[2;2;2]$ (see Figure \ref{fig:6^3_1}), has $u_+=3$, $u_0=4$, and $u_-=1$.  The coefficient $u_0$ exceeds the number of twist sites by 1, and $u_-$ exceeds the number of right-turning twist sites by 1, while $u_+$ accurately predicts the number of left-turning twist sites.
\end{enumerate}

 \begin{figure}[h]
   \centering
    \includegraphics[width=.5\textwidth]{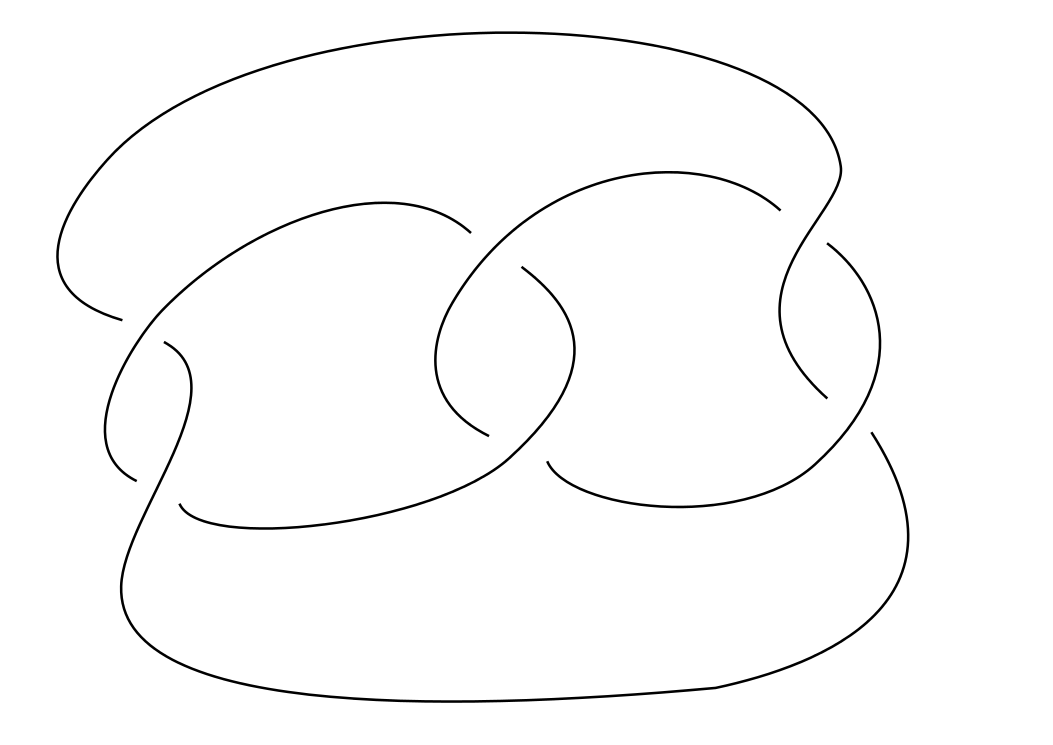}
 \caption{The three component link with Rolfsen notation $6^3_1$.}\label{fig:6^3_1}
 \end{figure}


\end{document}